\documentclass{cimart}

\DeclareMathOperator{\supp}{supp}
\DeclareMathOperator{\vol}{vol}

\newcommand\SL{\operatorname{SL}(2,\mathbb{R})}
\newcommand\GL{\operatorname{GL}(d,\mathbb{R})}
\newcommand\bN{\mathbb{N}}
\newcommand\bZ{\mathbb{Z}}
\newcommand\bR{\mathbb{R}}
\newcommand\cF{\mathcal{F}}
\newcommand\cM{\mathcal{M}}
\newcommand\cP{\mathcal{P}}
\newcommand\cQ{\mathcal{Q}}
\newcommand\cW{\mathcal{W}}
\newcommand\hFur{h^{\operatorname{Furs}}}

\VOLUME{34}
\YEAR{2026}
\ISSUE{2}
\NUMBER{10}
\DOI{https://doi.org/10.46298/cm.17138}
\licence{arXiv.org - Non-exclusive license to distribute}
\editinfo{December 17, 2025}{March 16, 2026}{Tiago Macedo, Jaqueline Mesquita, Mariel Saez and Rafael Potrie}
\acknowledgments{The authors thank Rafael Potrie for the invitation to write this survey. Also, thanks to Aaron Brown for the suggestions about Furstenberg entropy and to the anonymous referee for the useful comments. The authors were partially supported by CAPES-Finance Code 001. MP was supported by Instituto Serrapilheira, grant number Serra-R-2211-41879 and FUNCAP, grant AJC 06/2022.}

\title{Invariance principle in dynamical systems}

\authors{Karina Marin and Mauricio Poletti}

\authorinfo[K. Marin]{Federal University of Minas Gerais, Brazil}{kmarin@mat.ufmg.br}

\authorinfo[M. Poletti]{Federal University of Cear\'a, Brazil}{mpoletti@mat.ufc.br}

\abstract{%
	In this survey, we talk about what is known as \emph{Invariance Principle} in dynamical systems. It states that the disintegration of measures with zero center Lyapunov exponents admits some extra invariance by holonomies. We focus on explaining the basic definitions and ideas behind a series of results about the Invariance Principle and give some basic applications on how this is used in dynamical systems.
}



\begin{document}
	
	\section{Introduction}
	
	In order to study ergodic properties of a dynamical system $f\colon M\to M$ we need to understand the invariant probability measures. This means, probability measures $\mu$ such that $\mu(f^{-1}A)=\mu(A)$ for every measurable set $A\subset M$.
	
	When a diffeomorphism $f\colon M\to M$ is a partially hyperbolic map, one of the main features that it has is the existence of invariant foliations. By partially hyperbolic, we mean that on the tangent bundle of $M$ there exist invariant stable and unstable bundles, that respectively contracts and expands uniformly, and a center bundle that is dominated by the other ones, see Section~\ref{sec.holonomies} for the precise definition. The invariant foliations are tangent to these linear bundles, and are useful to study the dynamical behavior of $f$.
	
	We can relate foliations and invariant measures via disintegration, see Section~\ref{sec.disint} for the definitions and properties. Understanding these disintegrations gives a powerful tool to study ergodic properties of our dynamical systems. An example of this are the SRB measures of hyperbolic systems defined in \cite{Sin72, ruelle76, bowen2008}. These measures capture the statistical behavior of almost all points, with respect to the Lebesgue measure, and they are characterized by their disintegration along the unstable foliation.
	
	Another example of the importance of the disintegration along invariant foliations is the celebrated Ledrappier-Young formula \cite{LY85a} that relates Lyapunov exponents and metric entropy. The metric entropy of a measure-preserving dynamical system is a number that measures how chaotic the system is. The Lyapunov exponents measure the exponential expansion/contraction of the derivative. The formula states,
	$$
	h_\mu(f)=\int \sum_{\lambda_i(x)>0} \gamma_i(x)\lambda_i(x) d\mu(x)
	$$
	where $h_\mu(f)$ is the entropy of $f$ with respect to the invariant measure $\mu$ and $\lambda_i(x)$ are the Lyapunov exponents of $f$. The numbers $\gamma_i(x)\geq0$ which relate the two objects are defined using the disintegration of $\mu$. Specifically, $\gamma_1(x)$ is the Hausdorff dimension of the disintegration of the measure along the strongest unstable Pesin manifold on $x$. Moreover,  $\sum_{i=1}^k\gamma_i(x)$ is the dimension of the disintegration of the measure along the unstable Pesin manifold corresponding to the first $k$ exponents greater than zero.
	
	In this survey we will focus on partially hyperbolic dynamical systems that admit a center invariant foliation, and we will understand what the relation between entropy and the disintegration along this foliation is. The result we are going to explain shows that if the center foliation ``does not carry entropy'' then the disintegration of the measure is invariant under the holonomies of the stable and unstable foliations. This is known as the \emph{Invariance Principle}.
	
	In Sections~\ref{sec.disint} and \ref{sec.holonomies} we will give the basic definitions of disintegrations and holonomies. In Section~\ref{sec.IP}, we will state the Invariance Principle for partially hyperbolic dynamics and give the idea of the proof in a simpler case. We will give other versions of the Invariance Principle for random dynamics and cocycles in Sections~\ref{sec.Furstenberg} and \ref{sec.cocycles}. We finish with Section~\ref{sec.aplications} where we give applications of the Invariance Principle for studying genericity of non-zero center Lyapunov exponents, continuity of Lyapunov exponents, and rigidity of invariant measures.
	
	After finishing the manuscript, a survey on the same topic by François Ledrappier appeared online \cite{survey-Led}. Although both works address the same topic, their objectives differ:  Ledrappier’s survey focuses more on the proofs and provides a more detailed explanation of the techniques, whereas ours emphasizes applications to cocycles and partial hyperbolic dynamics.
	
	\section{Disintegration along foliations}\label{sec.disint}
	From now on, let $r\geq 1$ and $M$ be a compact $C^r$ manifold. A partition $\cP$ of $M$ is called measurable if it is a countable refinement of finite partitions.
	
	By Rokhlin Disintegration Theorem, given a probability measure $\mu$ and a measurable partition $\cP$, for $\mu$-almost every $x\in M$ there exists a probability measure $\mu^{\cP}_x$ such that
	\begin{itemize}
		\item $\mu^{\cP}_x(\cP(x))=1$, where $\cP(x)$ is the element of $\cP$ that contains $x$.
		\item $x\mapsto \mu^{\cP}_x(E)$ is measurable for every $E\subset M$ measurable.
		\item $\mu(E)=\int \mu^{\cP}_x(E)d\mu(x)$ for every $E\subset M$ measurable.
	\end{itemize}
	
	A \textit{foliation} of dimension $k$ on a manifold $M$  is  a  collection $\mathcal{F}$  of disjoint, connected, nonempty, immersed  $k$-dimensional $C^r$ submanifolds of $M$.
	In general, the partition by leaves of a foliation is not measurable, for example, the foliation of $\mathbb{T}^2$ by lines of irrational slope. Because of this, instead of taking the partition into global leaves, we do it locally.
	Suppose $M$ is a $d$-dimensional manifold. Let $\cF$ be a continuous foliation of $M$ with $C^r$ leaves of dimension $k\in \bN$. This means that for every $x\in M$, there exists a homeomorphism over its image
	$\Phi\colon B^k \times B^{d-k}\to M$, where $B^i$, $i\in \bN$, is the unit ball centered at the origin in $\bR^i$, such that
	$\Phi(0,0)=x$,
	and the restriction of $\Phi$ over $B^k\times\{y\}$ is a $C^r$ embedding depending continuously on $y$ and whose image is contained in some $\cF$ leaf.
	
	Given $\mu$ a probability measure, let $\mu_\Phi$ be the restriction of $\mu$ to the image of $\Phi$ defined as $\mu_\Phi(A)=\mu(A\cap\Phi(B^k\times B^{d-k}))$. Then, the partition $\cP_\Phi=\{\Phi(B^k\times\{y\}),y\in B^{d-k}\}$ is a measurable partition. By Rokhlin Disintegration Theorem, there exist $x\mapsto \mu^{\cF}_{\Phi,x}$ a disintegration of $\mu_\Phi$ with respect to $\cP_\Phi$.
	
	If $x\in M$ is contained in the image of two different foliation charts $\Phi_1$ and $\Phi_2$ then $\mu^{\cF}_{\Phi_1,x}=c \mu^{\cF}_{\Phi_2,x} $ for some constant $c>0$, see \cite[Lemma~3.2]{AVW}. This allows to define the disintegration of $\mu$ along $\cF$ up to a multiplicative constant $x\mapsto [\mu^\cF_x]$ where $[\mu^\cF_x]$ is an element of the projectivization of measures. This means $[\eta]=[\nu]$ if and only if it exists $c>0$ such that $\eta=c\nu$. When we write $\mu^{\cF}_x$, we mean a measure in the class of $[\mu^\cF_x]$.
	
	%\begin{remark}
	If the foliation is given by compact leaves and the quotient $\tilde{M}=M/\cF$ is a topological manifold, then the partition by global leaves is measurable. In this case, there is a canonical normalization such that $\mu^{\cF}_x(\cF(x))=1$. Also, as $\mu^{\cF}_x=\mu^{\cF}_y$ for $y\in \cF(x)$, we can see the disintegration as a function $\tilde{M}\ni \tilde{x}\mapsto \mu^{\cF}_{\tilde{x}}$.
	%\end{remark}

	\subsubsection*{Entropy along an expanding foliation}
	
	Consider $f\colon M\to M$ a diffeomorphism and let $\cF$ be an invariant expanding foliation. That is, for every $x\in M$, $f(\cF(x))=\cF(f(x))$ and there exists $0<\sigma<1$ and $C>0$ such that
	$$
	\|Df^{-n}(x)|T_x\cF\|\leq C\sigma^n.
	$$
	
	Given two partitions \( \cP\) and \( \cQ \) such that \( \cQ(x) \subset \cP(x) \) for all \( x \in M \), we write \( \cQ \geq \cP \).
	
	We denote by $H_\mu(\cQ|\cP)$ the conditional entropy of the partition $\cQ$ with respect to $\cP$ and the measure $\mu$. It is defined as
	$$
	H_\mu(\cQ|\cP)=-\int \log(\mu^\cP_x(\cQ(x)))d\mu.
	$$

	Given an $f$-invariant probability measure $\mu$, a measurable partition $\eta$ is said to be increasing and subordinate to \( \cF \) if it satisfies the following conditions:
	\begin{itemize}
		\item[(a)] \( \eta(x) \subseteq \cF(x) \) for \(\mu\)-almost every \( x \).
		\item[(b)] \( f^{-1}(\eta) \geq \eta \) (increasing property).
		\item[(c)] \( \eta(x) \) contains an open neighborhood of \( x \) in \( \cF(x) \) for \(\mu\)-almost every \( x \).
	\end{itemize}
	
	The existence of measurable partitions that are increasing and subordinate to \( \cF\) was proved in \cite{LS82}.

	\begin{definition}\label{foliation.entropy}
		Let $\mu$ be an $f$-invariant ergodic probability measure, and let $\eta$  be an increasing measurable partition subordinate to \( \cF\). We define the $\cF$ metric entropy of \( \mu \) as
		$$h^{\cF}_\mu(f) = H_\mu(f^{-1}\eta\mid \eta)=-\int \log(\mu^\eta_x(f^{-1}\eta (x))d\mu(x) .$$
	\end{definition}
	
	This definition does not depend on the specific increasing measurable partition subordinate to \( \cF \). For details, see \cite{LS82, HHW17, Yan21}. This measures the entropy of $f$ along the foliation $\cF$. If $f$ is uniformly hyperbolic (Anosov) and $\cF$ is the unstable foliation then $h_\mu(f)=h_\mu^\cF(f)$, in other words $\cF$ carries all the entropy of $f$.

	\section{Holonomies}\label{sec.holonomies}
	
	A $C^1$ diffeomorphism $f\colon M\to M$ is \textit{partially hyperbolic} if there exists a nontrivial splitting of the tangent bundle $$TM=E^{u}\oplus E^{c}\oplus E^{s}$$ invariant under the derivative map $Df$, a Riemannian metric $\left\| \cdot \right\|$ on $M$, and positive continuous functions $\vartheta$, $\widehat{\vartheta}$, $\gamma$, $\widehat{\gamma}$ with
	$$\vartheta < 1 < \widehat{\vartheta}^{-1} \quad  \text{and} \quad \vartheta< \gamma < \widehat{\gamma}^{-1}< \widehat{\vartheta}^{-1},$$ such that for any $x\in M$ and any unit vector $v\in T_{x}M$,
	\begin{equation}\label{ph}
		\begin{aligned}
			&\left\| Df_{x}(v) \right\|< \vartheta(x) \quad \quad \text{if} \; v\in E^{s}(x), \\
			\gamma(x) < &\left\| Df_{x}(v) \right\|< \widehat{\gamma}(x)^{-1}\quad  \text{if} \; v\in E^{c}(x),\\
			\widehat{\vartheta}(x)^{-1}< &\left\| Df_{x}(v) \right\| \quad \quad \quad \quad \quad \: \text{if} \; v\in E^{u}(x).
		\end{aligned}
	\end{equation}
	
	The stable and unstable bundles $E^{s}$ and $E^{u}$ are uniquely integrable, and their integral manifolds form two (continuous) foliations $W^{s}$ and $W^{u}$, whose leaves are immersed submanifolds of the same class of differentiability as $f$.
	%These foliations are called the \textit{strong-stable} and \textit{strong-unstable} foliations. They are invariant under $f$, in the sense that
	%$$f(W^{s}(x))=W^{s}(f(x)) \qquad \text{and}\qquad f(W^{u}(x))=W^{u}(f(x)),$$ where $W^{s}(x)$ and $W^{u}(x)$ denote the leaves of $W^{s}$ and $W^{u}$, respectively, passing through any $x\in M$.
	%For more information about partially hyperbolic diffeomorphisms we refer the reader to \cite{BDV,HPS,SH}.
	
	In general, there may not exist an invariant foliation tangent to $E^c$. We say that $f$ is \emph{dynamically coherent} if there also exist invariant foliations $W^{cs}$ and $W^{cu}$ tangent to the bundles $E^c\oplus E^s$ and $E^u\oplus E^c$ respectively: we call these foliations center-stable and center-unstable. The intersection of these foliations defines a center foliation $W^c$. Observe that the center and unstable foliations sub-foliate the center-unstable manifolds.
	
	Let $f\colon M\to M$ be a partially hyperbolic and dynamically coherent diffeomorphism. For $x\in M$, let $B^c_\delta(x)$ be the ball of radius $\delta$ inside $W^c(x)$ centered at the point $x$. Given $y\in W^{cu}(x)$ there exists a H\"older continuous map $h^u_{x,y}\colon B^c_{\delta}(x)\to W^c(y)$, for some $\delta>0$, such that $h^u_{x,y}(z)\in W^u(z)\cap W^c(y)$. These maps are called \emph{unstable holonomies}. Analogously, there exist \emph{stable holonomies} changing the unstable by the stable foliation.
	
	It is well known that when the map is $C^2$, these holonomies are absolutely continuous. This means that for every measurable $E\subset B^c_{\delta}(x)$, $vol^c_x(E)=0$ implies that $vol^c_y(h^u_{x,y}(E))=0$. Here, $vol^c_z$ is the volume measure on $W^c(z)$ induced by the Riemannian structure.
	
	If the functions in Equation $\eqref{ph}$ verify $\vartheta<\gamma \hat{\gamma}$ and $\hat{\vartheta}<\gamma \hat{\gamma}$, the map $f$ is said to be \emph{center bunched}. When $f$ is $C^2$, and center bunched, the unstable and stable holonomies are $C^1$ maps (see \cite{PSW}). If $f$ is $C^{1+\alpha}$, then we need a stronger bunching condition in order to obtain the regularity of the holonomies. We refer the reader to \cite{S} for this case.
	
	In order to simplify the notation, we denote by $\mu_x^c$ the disintegration of $\mu$ along the center foliation $W^c$.
	
	\begin{definition}\label{u-state}
		We say that an $f$-invariant measure $\mu$ has $u$-invariant center disintegration if there exists a $\mu$ total measure set $M'$ such that for every $x,y\in M'$, $y\in W^{cu}(x)$ we have $[{h^u_{x,y}}_*\mu^c_x|B^c_{\delta}(x)]= [\mu^c_y|Im (h^u_{x,y})]$.
		
		In other words, there exists $c>0$ such that, for every measurable $E\subset h^u_{x,y}(B^c_{\delta}(x))$, we have  $\mu^c_x({h^u_{x,y}}^{-1}E)= c\mu^c_y(E)$.
	\end{definition}
	
	We say that a map $f\colon N\times S\to N\times S$ is a \emph{skew product over}, $g\colon N\to N$, if $g\circ \pi=\pi\circ f$ where $\pi\colon N\times S\to N$ is the natural projection. If $f$ is a partially hyperbolic skew product over an Anosov diffeomorphism $g$, then $f$ admits global holonomies. More precisely, for $x,y\in N$ such that $x\in W_g^u(y)$, there exists a family of homeomorphisms $h^u_{x,y}\colon S\to S$ given by,
	\begin{equation}\label{eq.holonomy}
		h^u_{x,y}=\lim_{n\to \infty} (f^{-n}_{y})^{-1} \circ f^{-n}_{x}.
	\end{equation}
	Observe that in this case, $f$ is dynamically coherent and the center leaves are compact. Let $\tilde{\mu}=\pi_* \mu$, then, in this setting, $\mu$ has $u$-invariant center disintegration if there exists a $\tilde{\mu}$ total measure set such that $(h^u_{x, y})_*\mu^c_{x}=\mu^c_{y}$ for every $x,y$ in this set.

	%In the particular case of skew products there is a natural way to define holonomies. Let $f:M\to M$ be a skew product and $\tilde{f}:\tilde{M}\to \tilde{M}$ be the hyperbolic homeomorphism induced by the quotient of center leaves.
	%the convergence of this limit, and the regularity of this map depends on some bunching conditions. In the case of linear cocycles this is a linear map and the bunching condition, normally called Fiber-Bunched are...

	\section{The invariance principle}\label{sec.IP}
	
	Let $f\colon M\to M$ be a $C^1$ partially hyperbolic and dynamically coherent diffeomorphism. The map $f$ is said to act \emph{quasi-isometric} along the center if there exists $\delta_1,\delta_2>0$ such that $f^n(W^c_{\delta_1}(x))\subset W^c_{\delta_2}(f^n(x))$ for $n\in \bZ$.
	
	Partially hyperbolic diffeomorphisms with compact center leaves are quasi-isometric along the center. Another class of partially hyperbolic diffeomorphisms that are quasi-isometric along the center are the so called discretized Anosov flows \cite{Martinchich}. This class contains the perturbation of time one maps of geodesic flows in negative curvature and has non-compact center leaves. Other examples are the time one map of frame flows and isometric extensions of Anosov flows.

	We denote by $h^u_\mu(f)$ the unstable entropy of $f$, that is, the entropy of $f$ along the foliation given by unstable leaves as in Definition \ref{foliation.entropy}.
	
	\begin{theorem}
		Let $f\colon M\to M$ be a $C^1$ partially hyperbolic diffeomorphism which is dynamically coherent and acts quasi-isometric along the center. If $\mu$ is an ergodic invariant measure such that $h^u_\mu(f)=h_\mu(f)$, then the center disintegration of $\mu$ is $u$-invariant.
	\end{theorem}
	
	In what follows, we give the idea of the proof in the particular case when $f$ is a skew product.
	
	Recall that $f\colon N\times S\to N\times S$ and $g\colon N\to N$ is an Anosov diffeomorphism such that $g\circ \pi=\pi\circ f$, where $\pi\colon N\times S\to N$ is the natural projection.
	
	Let $\cW^u_g$ be the unstable foliation of $g$ and $\cW^{u}_f$ be the strong unstable foliation of $f$. It is known that the restriction of $\pi$ to $\cW^{u}_f(x)$ is a diffeomorphism over $\cW^u_g(\pi(x))$ (see, for example~\cite{burns2001recent}).
	
	Let $\eta_g$ be an increasing partition subordinated to $\cW^u_g$. Let $\eta_f$ be the partition given by $\eta_f(x)=\pi^{-1}(\eta_g(\pi(x)))\cap \cW^{u}_f(x)$. This is an increasing partition subordinated to $\cW^{u}_f$.
	
	The proof we are going to explain here was given in \cite[Theorem~A and Corollary~2.2]{tahzibi-yang-IP}, see the article for more details. The proof of the general case can be found in \cite{CP}.
	
	The main idea is to use the definition of unstable entropy and Jensen inequality for convex functions to find an inequality between unstable entropy and the full entropy, then use the fact that the equality in Jensen is only satisfied when the function is constant almost everywhere. This will give a condition on the disintegration that will be translated to $u$-invariance.

	\begin{proof}
		Let $\tilde{\mu}=\pi_* \mu$. Fix $z\in N$ and let $\widetilde{R}=\eta_g(z)$ and $R=\pi^{-1}(\widetilde{R})$. Observe that $g^{-1}(\eta_g)$ gives a countable partition of $\widetilde{R}$. We denote this partition by $\{\widetilde{Q}_i\}_{i\in \mathbb{N}}$ and define $Q_i:=\pi^{-1}(\widetilde{Q}_i)$.
		
		Given $t\in S$, let $\pi_t\colon \eta_f(z,t)\to \widetilde{R}$, be the restriction of $\pi$. This map is a homeomorphism. Write $R$ in the coordinates $\widetilde{R}\times S\ni [x,t]=(x,\pi^{-1}_t(x))$. Observe that in these coordinates $[x,t]$ and $[y,t]$ belong to the same $\cW^u_f$-leaf.
		
		Let $\mu_{R}$ be the measure on $R$ given by the disintegration of $\mu$ with respect to the partition $\pi^{-1}\eta_g$ and let  $\mu^{\eta_f}_t$ be the disintegration of $\mu_{R}$ into horizontals $[\widetilde{R},t]$.
		
		Define $\mu_{\widetilde{R}}=\pi_*\mu_{R}$ and $\nu_{R}$ be the projection of $\mu_{R}$ into $S$ (with coordinates $[\cdot , \cdot]$).
		
		By Jensen's inequality, we know
		$$
		\begin{aligned}
			\int_{S}-\mu^{\eta_f}_t(Q_i)\log \mu^{\eta_f}_t(Q_i)d\nu_{R}&\leq -\mu_{R}(Q_i)\log \mu_{R}(Q_i)\\
			&=-\mu_{\widetilde{R}}(\widetilde{Q}_i)\log \mu_{\widetilde{R}}(\widetilde{Q}_i),
		\end{aligned}
		$$
		with equality if and only if $\mu^{\eta_f}_t(Q_i)$ is constant $\nu_R$-almost everywhere.
		
		Summing on $i$, we get the entropy of the partition given by $f^{-1}\eta_f$ with respect to the partition $\eta_f$, so we have the inequality
		$$
		-\int_{R}\log \mu^{\eta_f}_{x,t}(f^{-1}\eta_f (x,t)) d\mu_{R}(x,t)\leq -\int_{\widetilde{R}}\log \mu_{\widetilde{R}}(g^{-1}\eta_g(x))d\mu_{\widetilde{R}}(x),
		$$
		with equality if and only if $\mu^{\eta_f}_{x,t}(f^{-1}\eta_f (x,t))=\mu_{\widetilde{R}}(g^{-1}\eta_g(x))$ for $\mu_{R}$-almost every $(x,t)$.
		
		Recall that $R$ depends on $z$, so integrating on $z$, we conclude
		$$
		h^u_{\mu}(f)\leq H_{\tilde{\mu}}(g^{-1}\eta_g|\eta_g),
		$$
		with equality if and only if $\mu^{\eta_f}_{x,t}(f^{-1}\eta_f (x,t))=\mu^{\eta_g}_{x}(g^{-1}\eta_g(x))$ for $\mu$-almost every point.
		
		By \cite{LS82,led84}, $h_{\tilde{\mu}}(g)=H_{\tilde{\mu}}(g^{-1}\eta_g|\eta_g)$, therefore we conclude that the equality holds and for $\mu$-almost every point,  thus $\mu^{\eta_f}_{x,t}(f^{-1}\eta_f (x,t))=\mu^{\eta_g}_{x}(g^{-1}\eta_g(x))$.
		
		Using that $h^u_\mu(f^n)=n h^u_\mu(f)$, we conclude that for every $n\in \mathbb{N}$, $$\mu^{\eta_f}_{x,t}(f^{-n}\eta_f (x,t))=\mu^{\eta_g}_{x}(g^{-n}\eta_g(x)).$$  Then, as diameter of $f^{-n}\eta_f(x,t)$ goes to zero, we know
		$\mu^{\eta_f}_{x,t}=\mu^{\eta_g}_{x}$.
		This implies that in the coordinates $[\cdot,\cdot]$, $\mu_{R}=\mu_{\widetilde{R}}\times \nu_R$.
		
		Observe that in coordinates $[\cdot,\cdot]$, the unstable holonomy is the identity on the second coordinate, so the $u$-invariance of the center disintegration is equivalent to $\mu_{R}$ being a product. This concludes the proof.
	\end{proof}
	
	\subsection{Product structure and continuous extension}
	
	The $u$-invariance condition on the disintegration is a measurable condition; it happens in a full measure set, and so normally this is not very useful because measurable functions can have a very complicated behavior and are defined only on a total measure subset. However, when we have both $s$-invariance and $u$-invariance and some extra condition on the measure, we can extend this disintegration to a continuous one on the support of the measure. This gives a very rigid condition that will have many consequences, as we will see in Section~\ref{s.applications}.
	
	\subsubsection*{Skew products over Anosov maps} First, we state and prove the result for the case of skew products, $f\colon N\times S\to N\times S$ over the Anosov map $g\colon N\to N$. Let $\widetilde{N}\subset N$ and $\cM(S)$ be the set of probability measures of $S$. A map $\widetilde{N}\ni x\mapsto \mu^c_{x}\in \cM(S)$ is continuous if it is continuous with respect to the weak$^*$ topology on $\cM(S)$.
	
	Recall that in this setting, an $f$-invariant measure $\mu$ has $su$-invariant center disintegration if there exists a $\tilde{\mu}$ full measure set $N^{s/u}$ such that
	$${h^{s/u}_{x,y}}_* (\mu^c_{x})=\mu^c_{y}\text{ for every } x, y\in N^{s/u}\text{ with }y\in W^{s/u}_g(x).$$
	
	Since $g$ is an Anosov map, for every $x\in N$ there exists $\delta>0$ such that the map
	\[\Phi_g\colon W^s_{\delta}(x)\times W^u_\delta(x)\to N\]
	is a homeomorphism over its image,  where $\Phi_g(y,z)$ is the unique point in $W^u_{loc}(y)\cap W^s_{loc}(z)$.
	
	\begin{definition}
		A $g$-invariant measure $\widetilde{\mu}$ has local product structure if for every $\Phi_g$ as before, $(\Phi_g^{-1})_*\widetilde{\mu}$ is equivalent to a measure of the form $\mu^s\times \mu^u$, with $\mu^{s/u}$ a measure on $W^{s/u}_\delta(x)$ .
	\end{definition}
	
	We remark that since $g$ is an Anosov map, then every equilibrium state associated to a H\"older continuous potential has local product structure.
	
	\begin{proposition}[Theorem 6, \cite{BGV}]\label{prop.product-structure}
		Let $\mu$ be an $f$-invariant measure with $su$-invariant center disintegration and such that $\widetilde{\mu}=\pi_* \mu$ has local product structure. Then, there exists a center disintegration of $\mu$ such that $\supp(\widetilde{\mu})\ni x\mapsto \mu^c_{x}$ is continuous and $su$-invariant.
	\end{proposition}
	
	\begin{proof}
		Let $N'=N^s\cap N^u$, where $N^s$ and $N^u$ are the full measure sets in the definition of $su$-invariance. Then, $\widetilde{\mu}(N')=1$.
		
		%By hypothesis there exists $N^{s}\subset N$ with $\mu(N^{s})=1$ such that if $\tilde{x}\in N^{s}$, $\tilde{y}\in N^{s}\cap W^s(\tilde{x})$ then ${h^s_{\tilde{x},\tilde{y}}}_* \mu^c_{\tilde{x}}=\mu^c_{\tilde{y}}$. Analogously there exists $N^u$, changing stable by unstable,
		
		Fix $x\in \supp(\widetilde{\mu})$. By the product structure given by $\Phi_g$, we use the local coordinates $W^s_\delta(x)\times W^u_\delta(x)$. Since $\widetilde{\mu}$ has local product structure, for $\mu^s$-almost every $y\in W^s_\delta(x)$, we have that $\mu^u((\{y\}\times W^u_\delta(x))\setminus N')=0$. Let $(y_0,z_0)\in (\{y_0\}\times W^u_\delta(x))\cap N'$.
		
		For $y\in W^s_\delta(x)$, define $$\nu_{(y,z_0)}={h^s_{(y_0,z_0),(y,z_0)}}_* \mu^c_{(x_0,y_0)}.$$ Observe that if $(y,z_0)\in N'$ then $\nu_{(y,z_0)}=\mu^c_{(y,z_0)}$. Now for $z\in W^u_\delta(x)$, define $$\nu_{(y,z)}={h^u_{(y,z_0),(y,z)}}_* \nu_{(y,z_0)}.$$ Therefore, if $(y,z_0)\in N'$ and $(y,z)\in N'$, then $\nu_{(y,z)}=\mu^c_{(y,z)}$. Observe that this set has full measure inside $W^s_\delta(x)\times W^u_\delta(x)$, therefore equality holds almost everywhere.
		
		Moreover, by the uniform continuity of the holonomies in $W^s_\delta(x)\times W^u_\delta(x)$, the disintegration $(y, z)\mapsto \nu_{(y, z)}$ is uniformly continuous and $u$-invariant. Exchanging the roles of the stable and the unstable manifolds, we can find a disintegration $(y, z)\mapsto \nu'_{(y, z)}$ which is continuous and $s$-invariant. Since both disintegrations coincide in a dense set of $\supp(\widetilde{\mu})\cap W^s_\delta(x)\times W^u_\delta(x)$, then continuity implies that they coincide on the whole set.
		
		We can repeat this argument for every $x\in \supp(\mu)$ and since any such disintegrations will coincide in the intersection, we can define a center disintegration of $\mu$ satisfying the desired properties.
	\end{proof}
	
	As a consequence, we get the following result.
	\begin{corollary}\label{th.product-structure}
		Let $f$ be a partially hyperbolic skew product. Suppose $\mu$ is an $f$-invariant ergodic measure such that $h^u_\mu(f)=h^s_\mu(f)=h_\mu(f)$ and $\widetilde{\mu}=\pi_* \mu$ has local product structure. Then, there exists a center disintegration of $\mu$ such that $\supp(\widetilde{\mu})\ni x\mapsto \mu^c_{x}$ is continuous and $su$-invariant.
	\end{corollary}
	
	\subsubsection*{Partially hyperbolic maps} Let $f\colon M\to M$ be a $C^1$ partially hyperbolic diffeomorphism which is dynamically coherent and acts quasi-isometric along the center.
	
	Since $f$ is a partially hyperbolic map, for every $x\in M$ there exists $\delta>0$ such that the map $\Phi_f\colon W^{cs}_{\delta}(x)\times W^u_\delta(x)\to M$, where $\Phi_f(y,z)$ is the unique point in $W^u_{loc}(y)\cap W^{cs}_{loc}(z)$, is a homeomorphism over its image.
	
	\begin{definition}
		An $f$-invariant measure $\mu$ has local $cs\times u$ product structure if for every $\Phi_f$ as before $(\Phi_f^{-1})_*\mu$ is equivalent to a measure of the form $\mu^{cs}\times \mu^u$, where $\mu^{cs/u}$ is a measure on~$W^{cs/u}_\delta(x)$.
	\end{definition}
	
	The local $cs\times u$ product structure is satisfied by classes of measures as equilibrium states and some $u$-Gibbs, see \cite{CP}.
	
	The following corollary is an extension of Corollary \ref{th.product-structure} to the general setting. The proof is a consequence of an analogous version of Proposition \ref{prop.product-structure}. See Section 6 of \cite{CP}.
	
	\begin{corollary}
		Let $\mu$ be an $f$-invariant ergodic measure such that $h^u_\mu(f)=h^s_\mu(f)=h_\mu(f)$ and $\mu$ has local $cs\times u$ product structure. Then, there exists a center disintegration of $\mu$ such that $\supp(\mu)\ni x\mapsto \mu^c_{x}$ is continuous and $su$-invariant.
	\end{corollary}
	
	The concepts behind the Invariance Principle trace back to the work of Furstenberg for random products of matrices \cite{Fur} and the measurable generalization of Ledrappier \cite{ledrappierIP}. In the following sections, we present different settings where the various versions of the Invariance Principle can be applied.

	\section{Random dynamics and Furstenberg entropy}\label{sec.Furstenberg}
	
	As before, let $M$ be a compact Riemannian manifold, and for $r\geq 1$, let $\operatorname{Diff}^r(M)$ be the space of $C^r$ diffeomorphisms of $M$.
	We consider a probability measure $\nu$ in $\operatorname{Diff}^r(M)$ with compact support and denote $\Sigma=\supp(\nu)^\mathbb{N}$. We further consider $\sigma\colon \Sigma\to \Sigma$ the shift map given by $\sigma(f_0,f_1,\dots)=(f_1,f_2,\dots)$.
	
	The random dynamical system associated to $\nu$ is defined as the skew product, $$
	F\colon \Sigma\times M\to \Sigma\times M\quad (\omega,x)=(\sigma(\omega),f_0(x)),
	$$ where $\omega=(f_0,f_1,\dots)$.

	A measure $\eta$ in $M$ is called $\nu$-stationary if
	\begin{equation}
		\eta=\int f_*\eta d\nu(f).
	\end{equation}
	Observe that $\eta$ being $\nu$-stationary is equivalent to $\nu^\bN \times \eta $ being $F$-invariant. We call $\nu$ the driving measure.
	
	For each $f\in \supp(\nu)$, define $Jf=\dfrac{d (f^{-1}_*\eta)}{d\eta}$. In other words, $f^{-1}_*\eta=Jf \eta +\eta_f$, where $\eta_f$ is singular with respect to $\nu$.
	
	\begin{definition}[see \cite{Fur}] The Furstenberg entropy is defined as
		$$
		\hFur(\nu,\eta)=-\int \int \log Jf(x) d\eta(x)d\nu(f).
		$$
	\end{definition}

	\begin{lemma}
		$\hFur(\nu,\eta)\geq 0$ with equality if, and only if, $f_*\eta=\eta$ for $\nu$-almost every $f$. Moreover, $\hFur(\nu,\eta)<\infty$ implies that $f_*\eta \ll \eta$ for $\nu$-almost every $f$, that is, $f_*\eta$ is absolutely continuous with respect to $\eta$.
	\end{lemma}

	\begin{proof}
		The first part is a direct consequence of the Jensen inequality.
		
		For the second part, take $A\subset \supp(\nu)\times M$ such that $Jf(x)=0$. This set has zero $\nu\times \eta$ measure. This implies that for $\nu$-almost every $f$, $Jf(x)>0$ for $\eta$-almost every $x\in M$. In particular,  we conclude that $\eta\ll f^{-1}_*\eta$ which finish the proof of the lemma.
	\end{proof}
	
	Given $\omega =(f_0,f_1,\dots)\in \Sigma$, we denote $f_\omega^n=f_{n-1}\circ \cdots \circ f_0$. If $\log\|Df\|\in L^1(\nu)$, then we can define the Lyapunov exponents of $F$ as the limits
	$$\lim_{n\to \infty}\frac{1}{n}\log \|Df^n_\omega(x)v\|,\quad v\in T_x M.$$
	By Oseledets Theorem \cite{Ose68}, the limit is well defined for $\nu^\mathbb{N}\times \eta$-almost every $(\omega,x)$. Let $\lambda_{min}(\omega,x)$ be the smallest Lyapunov exponent on $(\omega,x)$.

	The following result gives an inequality between the Lyapunov exponents and the Furstenberg entropy. It was proved by Ledrappier for projective linear cocycles \cite{ledrappierIP}, for $C^1$ random walks by Crauel \cite{crauel1}, and for skew products by Avila-Viana \cite{AV-IP}.
	
	\begin{proposition}
		Suppose that $\log\|Df\|\in L^1(\nu)$, then
		$$
		\hFur(\nu,\eta)\leq \dim(M)\int \min\{-\lambda_{min}(\omega,x),0\} d\eta(x)d\nu^\bN(\omega)
		$$
	\end{proposition}
	
	As a consequence, we get the following version of the Invariance Principle in the setting of random dynamical systems.
	\begin{theorem}
		For the random dynamical system with driving measure $\nu$ and stationary measure $\eta$, if all Lyapunov exponents are non-negative $\nu^{\bN}\times \eta$-almost everywhere, then $\eta$ is invariant for $\nu$-almost every $f$.
	\end{theorem}

	Malicet proved a version of this theorem for random walks of homeomorphisms in the circle, and he used this result to study synchronization \cite{malicet17}.

	\section{Cocycles}\label{sec.cocycles}
	\subsection{Linear cocycles} In the following, we consider $f\colon M\to M$ as an Anosov or a partially hyperbolic diffeomorphism and $A\colon M\to \GL$ an $\alpha$-H\"older continuous map. The pair $(f,A)$ is called a linear cocycle and defines the skew product, $$
	F_A\colon M\times \mathbb{R}^d\to M\times \mathbb{R}^d,\quad (x,v)\mapsto (f(x),A(x)v).
	$$
	The Lyapunov exponents of $(f,A,\mu)$ are defined as the limits, $$
	\lim_{n\to \infty} \frac{1}{n}\log \|A^n(x)v\|.
	$$
	As before, by Oseledets Theorem, if $\mu$ is an $f$-invariant measure, then the limit exists for \mbox{$\mu$-almost} every $x$.
	
	Moreover, if $\mu$ is ergodic, there exist numbers $\lambda_1(A)> \cdots >\lambda_k(A)$, with $k\leq d$, and a measurable decomposition $E_1(x)\oplus \cdots \oplus E_k(x)$, defined $\mu$-almost everywhere, such that $$
	\lambda_i(A)=\lim_{n\to \infty} \frac{1}{n}\log \|A^n(x)v\|,\quad \text{if }v\in E_i(x),\,1\leq i\leq k.$$
	
	%Let $C>0$ and $0<\gamma<1$ be such that for every $n\in \mathbb{N}$
	%$$
	%\max\{\|Df^n(x)\mid E^s\|,\|Df^{-n}(x)\mid E^u\|\}\leq C\gamma^n.
	%$$
	
	Denote by $\zeta$ the rate of contraction and expansion if $f$ is Anosov and $\zeta=\max\{\vartheta, \widehat{\vartheta}\}$ if $f$ is partially hyperbolic. Here $\vartheta$ and $\widehat{\vartheta}$ are the functions in Equation (\ref{ph}).
	
	The cocycle $(f,A)$ is said to be \emph{fiber-bunched} if there exists $N\in \mathbb{N}$ such that for every $x\in M$, \begin{equation}\label{eq.FB}
		\|A^N(x)\|\|A^N(x)^{-1}\|\zeta^{\alpha N} <1.
	\end{equation}
	
	This condition implies the existence of stable and unstable holonomies which in this case are linear isomorphisms $H_{x,y}^{s/u}\colon \mathbb{R}^d\to \mathbb{R}^d$. For $y\in W^s(x)$, the holonomies are defined as $H^s_{x,y}=\lim_{n\to \infty}A^n(y)^{-1}\circ A^n(x)$. The definition is analogous for points in the same unstable manifold.
	
	The cocycle $(f,A)$ induces a skew product map on the projective fiber bundle, defined as
	$$
	PF_A\colon M\times P\mathbb{R}^d\to M\times P\mathbb{R}^d, \quad (x,[v])\mapsto (f(x),[A(x)v]).
	$$
	
	Observe that $PF_A$ acts by diffeomorphism on the fibers. For any $PF_A$-invariant measure $m$, the Lyapunov exponents of $PF_A$ exists $(x,v)$-almost every point, as the limits $$\lim_{n\to \infty}\frac{1}{n}\log \|D(PF_A)^n_x([v])w\|,\quad w\in T_{[v]} P\mathbb{R}^d.$$
	
	We remark that if $A$ is a $C^1$ map, then the fiber-bunched condition (with $\alpha=1$) implies that $PF_A$ is a partially hyperbolic diffeomorphism and there is a relation between the Lyapunov exponents of $PF_A$ as projective cocycle and the center Lyapunov exponents of the partially hyperbolic map.  Moreover, if we allow the base map $f$ to be a shift map and $A$ to depend only on the zero coordinate, then we can relate the projective cocycle with the setting of random dynamics mentioned in Section~\ref{sec.Furstenberg}.
	
	The following relation holds for any $A$ continuous linear cocycle over $f$,
	$$
	\|A(x)\|^{-1} \|A(x)^{-1}\|^{-1}\leq \|D(PF_A)_x([v])w\|\leq \|A(x)\| \|A(x)^{-1}\|,
	$$
	for $w\in T_{[v]} P\mathbb{R}^{d}$. Therefore, if $\mu$ is an $f$-invariant ergodic measure and $m$ is an $PF_A$-invariant measure projecting to $\mu$, the absolute value of every Lyapunov exponent of $PF_A$ is bounded by $\lambda_1(A)-\lambda_k(A)$. In particular, if $k=1$, this implies that all the Lyapunov exponents of $PF_A$ are zero.
	
	If $(f,A)$ is a fiber-bunched cocycle, then $PF_A$ also admits stable and unstable holono\-mies, defined by $h^{s/u}_{x,y}=PH^{s/u}_{x,y}$. Then, for any measure $m$ which is $PF_A$-invariant, we can consider the disintegration along the fibers and extends the notion of $su$-invariance as in Definition \ref{u-state}.
	
	The following version of the Invariance Principle is a direct consequence of the measurable version of Ledrappier \cite{ledrappierIP} and was proven for Anosov maps in \cite{BGV} and for partially hyperbolic maps in \cite{ASV}.
	
	\begin{theorem}\label{IP-linear}
		Let $f\colon M \to M$ be a $C^{1}$ Anosov or partially hyperbolic diffeomorphism. Let $A\colon M\to GL(d,\mathbb{R})$ be an $\alpha$-H\"older linear cocycle such that $(f,A)$ is fiber-bunched.
		Suppose $\mu$ is an $f$-invariant ergodic measure such that $\lambda_{1}(A)=\lambda_{k}(A)$. Then, every $PF_A$-invariant measure $m$ projecting to $\mu$ admits a disintegration $\left\{m_{x} : x\in M \right\}$ along the fibers which is $su$-invariant.
	\end{theorem}
	
	\subsubsection*{Continuity of Lyapunov exponents for $d=2$}
	
	In the case that $A\colon M\to \SL$, there exist at most two different Lyapunov exponents, which satisfy $\lambda_1(A)\geq \lambda_2(A)$ and $\lambda_1(A)=-\lambda_2(A)$. In order to simplify the notation, we denote $\lambda(A)=~\lambda_1(A)$, $E^+(x)=E_1(x)$ and $E^-(x)=E_2(x)$.
	
	Fix $f\colon M \to M$ an Anosov or partially hyperbolic diffeomorphism and $\mu$ an \mbox{$f$-invariant} ergodic measure. Let $A\colon M\to \SL$ be a continuous map. We are interested in the continuity of the map $$A\mapsto \lambda(A), $$
	where we take the uniform topology $\|A-B\|=\sup_{x\in M} \|A(x)-B(x)\|$.
	
	The function $\lambda(A)$ is upper semi-continuous, and as $0\leq \lambda(A)$ we have that if $\lambda(A)=0$ then $A$ is a continuity point of the Lyapunov exponents.
	
	If $\lambda(A)>0$, then there are two ergodic $PF_A$-invariant measures defined as, $$m^+=\int_M \delta_{E^+(x)} d\mu\quad \text{and}\quad m^-=\int_M \delta_{E^-(x)} d\mu.$$ Observe that,
	\begin{equation}\label{eq.lambda-m}
		\begin{aligned}
			&\int_{M\times P \mathbb{R}^d} \log \|A(x)v\|d m^+(x,[v])=\lambda(A),\\
			&\int_{M\times P \mathbb{R}^d} \log \|A(x)v\|d m^-(x,[v])=-\lambda(A),
		\end{aligned}
	\end{equation}
	where $v$ is a unit vector in the class of $[v]$.  Moreover, every $PF_A$-invariant probability measure $m$ that projects to $\mu$ is a convex combination of $m^+$ and $m^-$.
	
	%From now on lets assume that all cocycles are fiber bunched with the same $N$ in \eqref{eq.FB}.
	
	Suppose $A$ is an $\alpha$-H\"older cocycle such that $(f,A)$ is fiber-bunched. If $y\in W^s(x)$, for every $n\in \mathbb{N}$, $$A^n(x)v=H^s_{f^n(x),f^n(y)} A^n(y)H^s_{x,y}v.$$ Therefore, $$
	\lim_{n\to \infty}\frac{1}{n}\log\|A^n(x)v\|=\lim_{n\to \infty}\frac{1}{n}\log\|A^n(y)H^s_{x,y}v\|.
	$$ Since $E^-(x)$ is defined as the unique subspace such that for $v\in E^-(x)$, $$\lim_{n\to \infty}\frac{1}{n}\log\|A^n(x)v\|=-\lambda(A),$$ we conclude that
	$E^-(y)=H^s_{x,y} E^-(x)$. We can repeat an analogous argument to study the invariance of $E^+$. As a consequence, we have the following result.
	\begin{lemma}\label{lem.m+}
		The measure $m^+$ admits a $u$-invariant disintegration along the fibers and $m^-$ an $s$-invariant disintegration.
	\end{lemma}
	
	We define the $C^{\alpha}$ topology in the space of $\alpha$-H\"older continuous maps $A\colon M\to \SL$ given by the distance, $$\|A-B\|_{\alpha}=\|A-B\|+ H_{\alpha}(A-B),$$ where $H_{\alpha}(A)$ denotes the H\"older constant of $A$.
	
	We remark that the set of fiber-bunched cocycles is open in this topology. Moreover, if $A$ is fiber-bunched and $\|A_j-A\|_{\alpha}\to 0$, then the holonomies of $A_j$ converges uniformly to the holonomies of $A$.
	
	The next proposition characterizes the discontinuity points of the function \mbox{$A\mapsto \lambda(A)$}. It first appears in \cite{AV-IP}.
	
	\begin{proposition}\label{thm.continuity} Fix $f$ to be an Anosov or partially hyperbolic diffeomorphism and $\mu$ an $f$-invariant ergodic measure. Let $A$ be an $\alpha$-H\"older continuous cocycle such that $(f,A)$ is fiber-bunched.
		
		If $A$ is a discontinuity point of the Lyapunov exponents for the $\alpha$-H\"older topology, then every $PF_A$-invariant probability measure $m$ that projects to $\mu$ has $su$-invariant disintegration along the fibers.
	\end{proposition}
	\begin{proof}
		Let $A_j$, $j\in \mathbb{N}$ be a sequence of cocycles such that $\|A_j-A\|_{\alpha}\to 0$ and the Lyapunov exponents do not converge, that is, $\lim_{j\to \infty} \lambda(A_j)\neq \lambda(A)$.
		
		Then, the cocycles $(f, A_j)$ are all fiber-bunched with the same $N$ in Equation \eqref{eq.FB}.  Let $m^+_j$ be the measure defined as before for $A_j$ if $\lambda(A_j)>0$; if $\lambda(A_j)=0$ we take $m^+_j$ to be any $PF_A$-invariant probability measure. Observe that, by Theorem \ref{IP-linear} and Lemma \ref{lem.m+}, in both cases the measure $m^+_j$ admits a $u$-invariant disintegration along the fibers.
		
		Up to taking a sub-sequence we can assume that $m^+_j$ converges in the weak-star topology to an $PF_A$-invariant measure $m$ that projects to $\mu$. Moreover, $m$ also admits $u$-invariant disintegration. This is a consequence of the continuity of the holonomies, since $m^+_j$ admits $u$-invariant disintegration for every $j$. The proof of this fact is not direct. We refer the reader to \cite{bbb} for the Anosov case and \cite{notasm} for the partially hyperbolic case.
		
		Using Equation \eqref{eq.lambda-m}, we conclude,
		$$
		\lim_{j\to \infty} \lambda(A_j)=\int_{M\times P \mathbb{R}^d} \log \|A(x)v\|d m(x,[v]).
		$$
		Recall that $m=a m^+ + b m^-$ with $a+b=1$, then
		$\lim_{j\to \infty} \lambda(A_j)=a \lambda(A)-b \lambda(A)$, which implies that $b> 0$.
		
		%The uniformity of $N$ in \eqref{eq.FB} implies that for fixed $y\in W^{s/u}(x)$ the map $H^{s/u}_{A_k,x,y}$ converges to $H^{s/u}_{A,x,y}$. As by Lemma~\ref{lem.m+} the measure $m^+_k$ has $u$-invariant center disintegration if $\lambda(A_k)>0$ and by the invariance principle the same is true if $\lambda(A_k)=0$, so the limit measure $m$ is $u$-invariant (for a proof of this fact see ...).
		
		Then, we can write $m^-=b^{-1}(m-a m^+)$. By Lemma \ref{lem.m+} and the observation above, we conclude that $m^-$ admits $su$-invariant disintegration. Analogously, we can repeat the argument exchanging the roles of $m^+$ and $m^-$, in order to conclude that $m^+$ also admits  $su$-invariant disintegration. As every $PF_A$-invariant measure is a convex combination of $m^+$ and $m^-$, we proved the theorem.
	\end{proof}

	\subsection{Smooth Cocycles}
	
	Theorem \ref{IP-linear} was extended to the non-linear context in \cite{AV-IP}. The authors proved a measurable version that generalizes \cite{ledrappierIP} and can be applied to different non-linear settings. Here we state the version for trivial fiber bundles over Anosov maps \cite{AV-IP} and over partially hyperbolic diffeomorphisms \cite{ASV}.
	
	Suppose $f\colon M \to M$ is a $C^{1}$ Anosov or partially hyperbolic diffeomorphism. Let $N$ be a compact manifold and $\mathcal{A}\colon M\to \mathit{Diff}^1(N)$ be a continuous map. Then, $\mathcal{A}$ defines a \emph{smooth cocycle}, $$\mathcal{F}_{\mathcal{A}}\colon M\times N \to M\times N, \quad (x,y)\mapsto (f(x), \mathcal{A}(x)(y)).$$
	
	Observe that the projective cocycle defined in the last section is an example of a smooth cocycle.
	
	For any $\mathcal{F}_{\mathcal{A}}$-invariant measure $m$, the Lyapunov exponents of $\mathcal{F}_{\mathcal{A}}$ exists $(x,\xi)$-almost every point, as the limits $$\lim_{n\to \infty}\frac{1}{n}\log \|D(\mathcal{F}_{\mathcal{A}})^n_x(\xi)v\|,\quad v\in T_{\xi} N.$$
	
	\begin{theorem}\label{IP-nonlinear}
		Let $f\colon M \to M$ be a $C^{1}$ Anosov or partially hyperbolic diffeomorphism. Let $\mathcal{A}\colon M\to \mathit{Diff}^1(N)$ be a continuous smooth cocycle admitting stable and unstable holonomies.
		
		Suppose $m$ is an $\mathcal{F}_{\mathcal{A}}$-invariant probability measure such that for $m$-almost every $(x,\xi)$ all the Lyapunov exponents are zero. Then, $m$ admits a disintegration $\left\{m_{x} : x\in M \right\}$ along the fibers that is $su$-invariant.
	\end{theorem}
	
	The version of the Invariance Principle stated in Section~\ref{sec.IP} can be seen as an extension of Theorem \ref{IP-nonlinear} since the hypothesis of compactness of the fibers is relaxed by the hypothesis of the map acting quasi-isometric along the center.

	\section{Some applications}\label{s.applications}\label{sec.aplications}
	\subsection{Positivity of Lyapunov exponents}

	\subsubsection*{Cocycles over Anosov maps}
	
	Define $\mathcal{G}^{\alpha}(f,d)$ as the set of $\alpha$-H\"older fiber-bunched cocycles with values in $GL(d,\mathbb{R})$. Recall that for a linear cocycle $A$, $\lambda_{1}(A)$ denotes the largest Lyapunov exponent.
	
	We have the following result from \cite{BGV} for $d=2$ and \cite{BV} for the general case.
	
	\begin{theorem}\label{positive}
		Let $f\colon M\to M$ be a $C^1$ Anosov diffeomorphism and $\mu$ be an $f$-ergodic measure with local product structure and full support.
		
		There exists $\mathcal{U}$, an open and dense subset of $\mathcal{G}^{\alpha}(f,d)$ in the $\alpha$-H\"older topology, such that for every $A\in\mathcal{U}$, $\lambda_1(A)>0$.
	\end{theorem}
	
	The hypothesis of fiber-bunched in the theorem above can be removed by the work in \cite{Almost}. Moreover, \cite{Almost} generalizes Theorem \ref{IP-linear} and Proposition \ref{prop.product-structure} to the setting of non-uniform hyperbolic maps in the base.
	
	We provide the idea of the proof of Theorem \ref{positive} for the case of $\SL$-valued cocycles. Denote $\mathcal{S}^{\alpha}$ the subset of  $\mathcal{G}^{\alpha}(f,2)$ defined by cocycles $A\colon M\to \SL$.
	
	Let $A\in \mathcal{S}^{\alpha}$. The cocycle $(f,A)$ is said to be \emph{pinching} if there exists $p\in M$, $n_p$-periodic such that $A^{n_p}(p)$ has two different eigenvalues. Let $e_1,e_2\in \mathbb{P}\mathbb{R}^2$ be the corresponding eigendirections.
	
	Moreover, $(f,A)$ is said to be \emph{twisting} if it is pinching and there exists an element $z\in W^s(p)\cap W^u(p)$ such that
	$$
	\{e_1,e_2\}\cap H^s_{z,p}\circ H^u_{p,z}(\{e_1,e_2\})= \emptyset.
	$$
	
	\begin{proposition}
		Let $A\in \mathcal{S}^{\alpha}$. If $(f,A)$ is pinching and twisting, then the projective cocycle $PF_A$ does not admit an $su$-invariant disintegration.
	\end{proposition}
	\begin{proof}
		Assume, by contradiction, that there exists $m$ $PF_A$-invariant probability measure projecting to $\mu$ and admitting an $su$-invariant disintegration along the fibers.
		
		By Proposition~\ref{prop.product-structure}, the disintegration $x\mapsto m_x$ is continuous and $su$-invariant. Moreover, since $m$ is $PF_A$-invariant, for $\mu$-almost every $x\in M$, $A(x)_* m_x=m_{f(x)}$ and the continuity of the disintegration implies that this relation can be extended to every $x\in M$.
		
		The pinching property implies that $A^{n_p}(p)\colon \mathbb{P}\mathbb{R}^2\to \mathbb{P}\mathbb{R}^2$ is a Morse-Smale map with one attracting and one repelling fixed point that coincide with $e_1$ and $e_2$, respectively. Then, we have that $m_p=a\delta_{e_1}+b\delta_{e_2}$ for some $a,b\geq 0$, $a+b=1$.
		
		Let $H=H^s_{z,p}\circ H^u_{p,z}$. The $su$-invariance implies that $H_* m_p=m_p$. Therefore,
		$$
		H(\{e_1,e_2\})=\{e_1,e_2\},
		$$
		which contradicts the twisting hypothesis, proving the proposition.
	\end{proof}
	
	As a consequence of Theorems~\ref{th.product-structure} and Theorem \ref{thm.continuity}, we conclude the following result. Recall that in the $\SL$ case, $\lambda_1(A)=-\lambda_2(A)$.
	
	\begin{corollary}
		Let $A\in \mathcal{S}^{\alpha}$. If $(f,A)$ is pinching and twisting, then  $A$ has non-zero Lyapunov exponents and it is a continuity point for the Lyapunov exponents.
	\end{corollary}

	\subsubsection*{Cocycles over partially hyperbolic maps}
	
	In order to generalize Theorem \ref{positive} to the setting of partially hyperbolic diffeomorphisms in the base, we need an analogous result for Proposition \ref{prop.product-structure}.
	
	One approach in this direction is Corollary \ref{th.product-structure} combined with the notions of pinching and twisting defined in the previous section. This type of result allows one to prove the positivity of Lyapunov exponents for symplectic linear cocycles over some classes of partially hyperbolic maps. This is done in \cite{notasm}.
	
	A different technique is to consider volume-preserving diffeomorphisms. Recall that a partially hyperbolic map $f$ is center-bunched if the functions in Equation (\ref{ph}) satisfy $\vartheta<\gamma \hat{\gamma}$ and $\hat{\vartheta}<\gamma \hat{\gamma}$.
	
	Given two points $x,y\in M$, $x$ is \textit{accessible} from $y$ if there exists a path that connects $x$ to $y$, which is a concatenation of finitely many subpaths, each of which lies entirely in a single leaf of $W^u$ or a single leaf of $W^s$. We call this type of path an \textit{su-path}. This defines an equivalence relation, and we say that $f$ is \textit{accessible} if $M$ is the unique accessibility class.
	
	\begin{theorem}[Theorem A, \cite{ASV}]\label{ph.positive}
		Let $f$ be a $C^2$ partially hyperbolic volume-preserving diffeomorphism which is center-bunched and accessible, and $\mu$ be a measure in the Lebesgue class.
		
		There exists $\mathcal{U}$, an open and dense subset of $\mathcal{G}^{\alpha}(f,d)$ in the $\alpha$-H\"older topology, such that for every $A\in\mathcal{U}$, $\lambda_1(A)>0$.
	\end{theorem}
	
	%Suppose $m$ is a $PF_A$-invariant measure such that $m=\pi_* \mu$ and $m$ admits an $su$-invariant disintegration along the fibers. Then, there exists a disintegration of $m$ such that $x\mapsto m_x$ is continuous and $su$-invariant.
	
	The proof of Theorem \ref{ph.positive} is a consequence of Theorem D of \cite{ASV} where the authors established that in that setting, an $su$-invariant disintegration along the fibers can be upgraded to a disintegration which is continuous and $su$-invariant. Moreover, they replace the notion of pinching and twisting by properties of $su$-loops defined by the stable and unstable manifolds in the base.
	
	\subsubsection*{Center Lyapunov exponents}
	
	Since Theorem \ref{IP-linear} also holds true for general vector bundles (not necessarily trivial), a natural question is whether we can obtain results about positivity of the Lyapunov exponents by applying the Invariance Principle to the derivative cocycle.
	
	In particular, if $f$ is a $C^r$ partially hyperbolic diffeomorphism, we are interested in the center Lyapunov exponents, that is, the Lyapunov exponents of the cocycle $Df\vert E^c$.
	
	In the following, we consider partially hyperbolic maps with a 2-dimensional center bundle. Fixed an $f$-invariant measure $\mu$, we denote the center Lyapunov exponents by $\lambda^c_1(f,x)$ and $\lambda^c_2(f,x)$, and $\lambda^c_1(f)=\int \lambda^c_1(f,x)\, d\mu.$
	
	The main difficulty in order to obtain results analogous to the ones in the previous sections is that for the derivative cocycle, the dynamics in the base and the dynamics in the fibers are related, and we are not able to perturb one of them without modifying the other.
	
	The notions of pinching and twisting can be extended to this setting to obtain the following result.
	
	\begin{theorem}\cite{OP}
		For any $r>1$, among the volume preserving, $C^r$ partially hyperbolic skew products with 2-dimensional center that are center bunched, there exists a $C^1$-dense and $C^r$-open subset of diffeomorphisms verifying the following: if $f$ belongs to this subset, then $\lambda^c_1(f)>0$.
	\end{theorem}

	An analogous result can be concluded using the approach of \cite{ASV} and perturbative techniques.
	
	\begin{theorem}\cite{M, LMY}
		Fix $r\geq 2$. Let $f$ be a $C^r$ partially hyperbolic symplectic diffeomorphism with a 2-dimensional center bundle. Assume that $f$ is accessible, center-bunched, and the set of periodic points is non-empty. Then, $f$ can be $C^r$-approximated by $C^r$-open sets of symplectic diffeomorphisms with non-zero center Lyapunov exponents.
	\end{theorem}
	
	\subsubsection*{Random dynamics}
	
	For a random product of diffeomorphisms, we mention the following result, which also uses the Invariance Principle as part of the argument.
	
	\begin{theorem}\cite{Barrientos-Malicet}
		Given $r\geq 1$ and $k\geq 2$, consider conservative $C^r$-diffeomorphisms $f_1,\dots, f_{k-1}$ of $M$
		such that the Lebesgue measure is ergodic for the group generated by these maps. Then there is a
		$C^1$-open and $C^r$-dense set of conservative $C^r$-diffeomorphisms such that for any $f_k$ in this set, the Lyapunov exponents of the random product of $f_1,\dots, f_k$ have non-zero exponents.
	\end{theorem}

	\subsection{Rigidity of measures}
	The Invariance Principle is very useful to prove some rigidity results for systems that have measures with zero center Lyapunov exponents. We mention some of these applications.
	
	\subsubsection*{Skew products of rotations}
	Let $r>1$ and $g\colon N\to N$ be a diffeomorphism such that $\Lambda\subset N$ is a $g$-invariant compact hyperbolic set with local product structure. Let $\theta\colon N\to \mathbb{R}$ be a $C^r$ function.
	
	We define a skew product $f\colon N\times S^1\to N\times S^1$ to be a $C^r$, skew product that projects on $g$, in other words $f(x,t)=(g(x),t+\theta(x))$.
	
	Let $x,y\in \Lambda$ be in the same stable manifold. Using an identity analogous to Equation \eqref{eq.holonomy} for the stable holonomy, it is easy to see that the stable holonomy is defined as,
	\begin{equation}\label{eq.holonomy.rot}
		h^s_{x,y}(t)=t+\theta^s_{x,y},\,\text{ where }\theta^s_{x,y}=\sum_{k=0}^\infty \theta(f^k(x))-\theta(f^k(y)).
	\end{equation}
	
	Let $\tilde{\mu}$ be a $g$-invariant measure such that $\supp(\tilde{\mu})=\Lambda$ and $\tilde{\mu}$ has local product structure. If $\pi\colon N\times S^1\to N$ is the projection on the first coordinate, we want to answer the question of how many $f$-invariant measures $\mu$ such that $\pi_*\mu=\tilde{\mu}$ exist?
	
	As before, we call $\gamma=(x_0,x_1,x_2,\dots,x_k)$ an $su$-path if $x_i\in \Lambda$ and $x_{i+1}\in \Lambda$ are in the same stable or unstable manifold.
	For each $\gamma$ we define
	$$
	H_\gamma=  h^{\sigma_k}_{x_{k-1},x_k}\dots\circ h^{\sigma_1}_{x_0,x_1},
	$$
	where $\sigma_i$ is such that $x_{i-1}\in W^{\sigma_i}_g(x_i)$. By Equation \eqref{eq.holonomy.rot}, we know,
	$$
	H_\gamma(t)=t+\theta_\gamma \text{ where }\theta_\gamma= \sum_{i=1}^k \theta^{\sigma_i}_{x_{i-1},x_i}.
	$$
	
	Fix $z\in \Lambda$, we say that an $su$-path is an $su$-loop over $z$ if $x_0=x_k=z$. Define
	$$
	G_z=\{\theta_\gamma:\theta\text{ a }su\text{-loop over }z\}.
	$$
	Observe that $G_z$ is a group. Moreover, this group does not depend on $z$, we denote it by $G$.
	
	\begin{proposition}\label{prop.group.G}
		Let $m$ be the Lebesgue measure on $S^1$ and $\mu$ be an $f$-invariant measure such that $\pi_*\mu=\tilde{\mu}$.
		Either,
		\begin{itemize}
			\item $\mu=\tilde{\mu}\times m$, or
			\item There exists $n,\ell\in \mathbb{N}$ such that $G=\{\frac{1}{n},\dots, \frac{n-1}{n}\}$, and for every $p=f^{n_p}(p)$,
			\[\sum_{j=0}^{n_p-1}\theta_{f^j(p)}=1\mod \frac{1}{\ell}.\]
		\end{itemize}
	\end{proposition}
	
	\begin{proof}
		Any measure $\mu$ has zero Lyapunov exponent along the fiber. Then by Corollary~\ref{th.product-structure}, the disintegration along the fibers $\Lambda\ni x\mapsto \mu_x$ has a continuous extension to $\Lambda$ that is $su$-invariant.
		
		For each $\theta$, let $R_\theta$ be the rotation defined by $R_{\theta}(t)=t+\theta$. Then, $(R_\theta)_* \mu_x =\mu_x$ for every $\theta\in G_x$. Since this condition is closed in $\theta$, we know that it is satisfied for every $\theta\in \overline{G_x}$.
		
		The first part of the proposition follows from the fact that $\overline{G_x}$ is either $S^1$ or finite.
		
		Suppose the first case does not happen. Observe that for every $z\in \Lambda$, $\mu_z\neq m$. If this were not the case, then, by the holonomy invariance, $\mu_x=m$ for every $x\in \Lambda$.
		
		Now take $p\in\Lambda$ such that $f^{n_p}(p)=p$, let $\theta^{n_p}_p=\sum_{j=0}^{n_p-1}\theta_{f^j(p)}$. Since $(R_{\theta^{n_p}_p})_*\mu_p=\mu_p$, we conclude that $\theta^{n_p}_p=\frac{m_p}{\ell_p}$ for some $\ell_p,m_p\in \mathbb{N}$.
		
		We need to prove that $\ell_p$ is uniformly bounded. Assume there exists $p_j\to p$ periodic point such that $\ell_{p_j}\to \infty$. Take $\theta \in \mathbb{R}\setminus \mathbb{Q}$ and $q_j=\frac{m'_j}{\ell_{p_j}}\to \theta$. By the invariance, we know $(R_{q_j})_* \mu_{p_j}=\mu_{p_j}$. Letting $j\to \infty$ and using that $x\mapsto \mu_x$ is continuous we conclude that
		$ (R_\theta)_* \mu_p=\mu_p$. This implies that $\mu_p=m$, which is a contradiction.
	\end{proof}
	
	\begin{theorem}
		Let $\mu$ be an $f$-invariant measure that satisfies $\pi_*\mu=\tilde{\mu}$. Then, either $\mu=\tilde{\mu}\times m$, or there exists $n\in \mathbb{N}$ such that $\theta_x\mod \frac{1}{n}$ is cohomologous to a rational constant when restricted to $\Lambda$.
	\end{theorem}
	\begin{proof}
		Assume that $\mu$ is not equal to $\tilde{\mu}\times m$. Let $S:=S^1\mod \frac{1}{n}$ and define $\theta'\colon \Lambda\to S$ by $\theta'(x)=\theta_x\mod \frac{1}{n}$.
		
		Fix $p\in \Lambda$, a fixed point, and define $\varphi\colon \Lambda \to S$ as $\varphi(x)=\theta_\gamma$ where $\gamma$ is an $su$-path from $p$ to $x$. This map is well-defined because for any other $su$-path $\gamma'$, $\theta_\gamma-\theta_{\gamma'}=\theta_{\gamma\wedge -\gamma'}\in G_p$, where  $\gamma\wedge -\gamma'$ is the $su$-loop concatenating $\gamma$ with $\gamma'$ in the reverse order. Then, by Proposition~\ref{prop.group.G}, $\theta_\gamma-\theta_{\gamma'}\mod \frac{1}{n}=0$.
		
		By the definition of the holonomies, $\theta_x=\theta_{f(\gamma)}-\theta_\gamma+\theta_p$ where $\gamma$ is an $su$-path from $p$ to $x$. Taking the equality mod $\frac{1}{n}$, we conclude that  $\theta'(x)=\varphi(f(x))-\varphi(x)+\theta'(p)$.
		Again by Proposition~\ref{prop.group.G}, $\theta'(p)\in \mathbb{Q}$.
	\end{proof}
	
	\subsubsection*{Accessibility implies smoothness}
	
	In the following, we consider a more general setting.
	
	\begin{theorem}\label{thm.lebesgue}
		Let $f\colon M\to M$ be a $C^r$, $r>1$, partially hyperbolic center-bunched and accessible diffeomorphism which is dynamical coherent and acts quasi-isometric along the center.
		
		If $\mu$ is an ergodic $f$-invariant measure, fully supported with $cs\times u$ product structure and all center Lyapunov exponents equal to zero, then $\mu^c_x$ is absolutely continuous with respect to volume along $W^c(x)$.
	\end{theorem}
	\begin{proof}
		By the hypothesis, we can apply Corollary~\ref{th.product-structure}.
		Recall that $\mu^c_x$ is well-defined up to a constant, fix one. Therefore, $\mu^c_x=\varphi \vol^c + \vol^c_\perp$, where $\vol^c$ is the volume measure along $W^c$, $\varphi$ is a local $L^1$ function and $\vol^c_\perp$ is singular with respect to $\vol^c$.
		
		Recall also that for $\vol^c_\perp$-almost every $z\in W^c(x)$, $\frac{d\mu^c}{d\vol^c}(z)=\infty$, and for $\vol^c$-almost every $y\in W^c(x)$, $\frac{d\mu^c}{d\vol^c}(y)<\infty$.
		
		Take $y,z\in W^c(x)$, by accessibility, there exists an $su$-path, $\gamma$, from $z$ to $y$. This defines a composition of holonomies $h_\gamma\colon B^c_\delta(z)\to W^c(x)$ such that $h_\gamma(z)=y$. By the $su$-invariance of $\mu^c_x$ and the absolute continuity of the stable and unstable holonomies \cite{PSW}, there exists $K\geq 1$ such that
		$$
		K^{-1}\frac{\mu^c(h(B^c_\varepsilon(z)))}{\vol^c(h(B^c_\varepsilon(z))) }\leq\frac{\mu^c(B^c_\varepsilon(z))}{\vol^c(B^c_\varepsilon(z)) }   \leq K \frac{\mu^c(h(B^c_\varepsilon(z)))}{\vol^c(h(B^c_\varepsilon(z))) }.
		$$
		Since $f$ is center-bunched, we know that the holonomies are Lipchitz, so making $\epsilon\to 0$, we conclude that,
		$$
		K^{-1}\frac{d\mu^c}{d\vol^c}(y)\leq\frac{d\mu^c}{d\vol^c}(z)   \leq K \frac{d\mu^c}{d\vol^c}(y).
		$$
		In particular, there cannot exist $z$ such that $\frac{d\mu^c}{d\vol^c}(z)=\infty$, because it contradicts the condition $\frac{d\mu^c}{d\vol^c}(y)<\infty$ for $\vol^c$-almost every $y\in W^c(x)$. This implies that $\vol^c_\perp=0$. Moreover, if there exists $y$ such that $\frac{d\mu^c}{d\vol^c}(y)=0$, it implies that $\mu^c=0$ which is a contradiction. Then, $0<\frac{d\mu^c}{d\vol^c}(z)<\infty$.
	\end{proof}

	\subsubsection*{Physical measures}
	
	Given a diffeomorphism $f\colon M\to M$, an $f$-invariant probability measure $\mu$ is called \emph{physical} if its basin $$
	B(\mu)=\{x\in M:\lim_{n\to \infty}\frac{1}{n}\sum_{j=1}^n \varphi(f^j(x))=\int \varphi d\mu, \quad\forall\, \varphi\colon M\to \mathbb{R}\in C^0(M,\mathbb{R})\},
	$$ has positive Lebesgue measure.
	
	One of the first examples of physical measures for non-conservative dynamical systems is the SRB measures \cite{Sin72,ruelle76,bowen2008} of hyperbolic systems. The construction of SRB measures was generalized for partially hyperbolic diffeomorphisms. These measures are called $u$-Gibbs measures. An invariant probability measure is a $u$-\emph{Gibbs} if its $u$-disintegration is absolutely continuous with the Lebesgue measure along the unstable leaves.
	For partially hyperbolic diffeomorphisms, every physical measure is a $u$-Gibbs, see \cite{BDV}, but not every $u$-Gibbs is physical.
	
	Theorem~\ref{thm.lebesgue} provides a way to find physical measures that are non-hyperbolic. This appears in \cite{Viana-Yang} and \cite{CP}.
	\begin{theorem}
		Let $f\colon M\to M$ be a $C^r$, $r>1$, partially hyperbolic center-bunched and accessible diffeomorphism which is dynamical coherent and acts quasi-isometric along the center.
		
		Suppose $f$ has minimal $u$-foliation and such that the $cs$-foliation is absolutely continuous. Then, if $\mu$ is a $u$-Gibbs measure with all center exponents equal to zero, $\mu$ is physical.
	\end{theorem}
	\begin{proof}
		By the minimality of the $u$-foliation, we have that $\mu$ is fully supported. The $cs$-foliation being absolutely continuous implies that $\mu$ has $cs\times u$ product structure. Therefore, we can apply Theorem~\ref{thm.lebesgue} to conclude that the center disintegration is absolutely continuous with the volume measure along the center. The fact that the measure is absolutely continuous with respect to Lebesgue along the center and unstable implies that the disintegration of the measure along the whole center unstable manifold is volume. Then, using the absolute continuity of the stable manifold, we conclude that $\mu$ is physical.
	\end{proof}

	\subsubsection*{Anosov Flows}
	One of the first examples of partially hyperbolic diffeomorphisms are the time one maps, $\varphi_1$, of a geodesic flow, $\varphi_t\colon M\to M$, in a negative curvature manifold $M$. As being partially hyperbolic is an open condition, any diffeomorphism $C^1$ close to $\varphi_1$ is also partially hyperbolic.
	
	A natural question that arises is ``when is this perturbation actually the time one map of a flow?''
	Avila, Viana and Wilkinson~\cite{AVW} study this problem for conservative perturbations. They prove that either the center disintegration of the Lebesgue measure is atomic or $f$ is the time one map of an Anosov flow.
	
	For non-conservative dynamics, Crovisier and Poletti \cite{CP} prove that if $f$ has minimal \mbox{$u$-foliations}, then either $f$ has two hyperbolic measures of maximal entropy or $f$ has a unique measure of maximal entropy, and it is the time one map of a topological Anosov flow.
	
	Both results use a version of the Invariance Principle. Let us give the idea of how to prove these results.
	
	We call a partially hyperbolic map $f\colon M\to M$ a \emph{Discretized Anosov flow (DAF)} if there exists an Anosov flow $\varphi_t\colon M\to M$ such that $\frac{d\varphi_t}{dt}(x)|_{t=0}\in E^c_x$ and a continuous function $\tau\colon M\to (0,\infty)$ such that $f(x)=\varphi_{\tau(x)}(x)$. See \cite{Martinchich} for more details about this class. The perturbation of the time one map of an Anosov flow is a Discretized Anosov flow.
	
	\begin{theorem}
		Let $f\colon M\to M$ be a $C^r$, $r>1$, DAF, accessible such that there exists an invariant probability measure $\mu$ fully supported with zero center Lyapunov exponent and $cs\times u$ product structure. Then, there exists a topological Anosov flow $\psi_t\colon M\to M$, smooth along the center, such that $\psi_1(x)=f(x)$.
	\end{theorem}
	\begin{proof}
		By Theorem~\ref{thm.lebesgue}, we know that $\mu^c_x$ is absolutely continuous with respect to Lebesgue.
		For $y\in W^c(x)$, such that $\varphi_s(x)=y$, $s>0$, define $[x,y)=\{\varphi_t(x);0\leq t<s\}$.
		In this case, there exists a canonical normalization such that $\mu^c_x[x,f(x))=1$.
		It is not difficult to prove that this normalization is $f$ and $su$-invariant. Then, we actually have a family of measures $x\mapsto \mu^c_x$ that varies continuously.
		
		Define $\psi_t(x)$ such that $\mu^c_x([x,\psi_t(x)))=t$. Take $Z_0$ a continuous unit vector field tangent to $E^c$, in the same direction as $\frac{d\varphi_t(x)}{dt}|_{t=0}$. Define $\Delta(x)=\frac{d\mu^c_x}{d\vol^c}(x)$, using the same ideas of the proof of Theorem~\ref{thm.lebesgue}, we can prove that $\Delta$ is well defined and continuous. See \cite{AVW} or \cite[Corollary~6.8]{CP} for the details. Then, $\psi_t$ is the flow with vector field \[Z(x)=\Delta(x)^{-1} Z_0(x).\qedhere\]
	\end{proof}

	%%% REFERENCES %%%
		

{\small
		\begin{thebibliography}{10}
			
			\bibitem{ASV}
			A.~Avila, J.~Santamaria, and M.~Viana.
			\newblock Holonomy invariance: rough regularity and applications to {L}yapunov
			exponents.
			\newblock {\em Ast\'erisque}, (358):13--74, 2013.
			
			\bibitem{AV-IP}
			A.~Avila and M.~Viana.
			\newblock Extremal {L}yapunov exponents: an invariance principle and
			applications.
			\newblock {\em Invent. Math.}, 181(1):115--189, 2010.
			
			\bibitem{AVW}
			A.~Avila, M.~Viana, and A.~Wilkinson.
			\newblock Absolute continuity, {L}yapunov exponents and rigidity {I}: geodesic
			flows.
			\newblock {\em J. Eur. Math. Soc. (JEMS)}, 17(6):1435--1462, 2015.
			
			\bibitem{bbb}
			L.~Backes, A.~Brown, and C.~Butler.
			\newblock Continuity of {L}yapunov exponents for cocycles with invariant
			holonomies.
			\newblock {\em J. Mod. Dyn.}, 12:223--260, 2018.
			
			\bibitem{Barrientos-Malicet}
			P.~G. Barrientos and D.~Malicet.
			\newblock Extremal exponents of random products of conservative
			diffeomorphisms.
			\newblock {\em Math. Z.}, 296(3-4):1185--1207, 2020.
			
			\bibitem{BDV}
			C.~Bonatti, L.~J. D\'iaz, and M.~Viana.
			\newblock {\em Dynamics beyond uniform hyperbolicity}, volume 102 of {\em
				Encyclopaedia of Mathematical Sciences}.
			\newblock Springer-Verlag, Berlin, 2005.
			\newblock A global geometric and probabilistic perspective, Mathematical
			Physics, III.
			
			\bibitem{BGV}
			C.~Bonatti, X.~G\'omez-Mont, and M.~Viana.
			\newblock G\'en\'ericit\'e{} d'exposants de {L}yapunov non-nuls pour des
			produits d\'eterministes de matrices.
			\newblock {\em Ann. Inst. H. Poincar\'e{} C Anal. Non Lin\'eaire},
			20(4):579--624, 2003.
			
			\bibitem{BV}
			C.~Bonatti and M.~Viana.
			\newblock Lyapunov exponents with multiplicity 1 for deterministic products of
			matrices.
			\newblock {\em Ergodic Theory Dynam. Systems}, 24(5):1295--1330, 2004.
			
			\bibitem{bowen2008}
			R.~Bowen.
			\newblock {\em Equilibrium states and the ergodic theory of {A}nosov
				diffeomorphisms}, volume 470 of {\em Lecture Notes in Mathematics}.
			\newblock Springer-Verlag, Berlin, revised edition, 2008.
			\newblock With a preface by David Ruelle.
			
			\bibitem{burns2001recent}
			K.~Burns, C.~Pugh, M.~Shub, and A.~Wilkinson.
			\newblock Recent results about stable ergodicity.
			\newblock In {\em Smooth ergodic theory and its applications ({S}eattle, {WA},
				1999)}, volume~69 of {\em Proc. Sympos. Pure Math.}, pages 327--366. Amer.
			Math. Soc., Providence, RI, 2001.
			
			\bibitem{crauel1}
			H.~Crauel.
			\newblock Extremal exponents of random dynamical systems do not vanish.
			\newblock {\em J. Dynam. Differential Equations}, 2(3):245--291, 1990.
			
			\bibitem{CP}
			S.~Crovisier and M.~Poletti.
			\newblock Invariance principle and non-compact center foliations.
			\newblock arXiv:2210.14989.
			
			\bibitem{Fur}
			H.~Furstenberg.
			\newblock Noncommuting random products.
			\newblock {\em Trans. Amer. Math. Soc.}, 108:377--428, 1963.
			
			\bibitem{HHW17}
			H.~Hu, Y.~Hua, and W.~Wu.
			\newblock Unstable entropies and variational principle for partially hyperbolic
			diffeomorphisms.
			\newblock {\em Adv. Math.}, 321:31--68, 2017.
			
			\bibitem{ruelle76}
			A.~V. Kocergin.
			\newblock Nondegenerate saddles, and the absence of mixing.
			\newblock {\em Mat. Zametki}, 19(3):453--468, 1976.
			
			\bibitem{led84}
			F.~Ledrappier.
			\newblock Propri\'et\'es ergodiques des mesures de {S}ina\"i.
			\newblock {\em Institut des Hautes \'Etudes Scientifiques. Publications
				Math\'ematiques}, (59):163--188, 1984.
			
			\bibitem{ledrappierIP}
			F.~Ledrappier.
			\newblock Positivity of the exponent for stationary sequences of matrices.
			\newblock In {\em Lyapunov exponents ({B}remen, 1984)}, volume 1186 of {\em
				Lecture Notes in Math.}, pages 56--73. Springer, Berlin, 1986.
			
			\bibitem{survey-Led}
			F.~Ledrappier.
			\newblock Invariance principle and dynamics, 2026.
			
			\bibitem{LY85a}
			F.~Ledrappier and L.-S. Young.
			\newblock The metric entropy of diffeomorphisms. {I}. {C}haracterization of
			measures satisfying {P}esin's entropy formula.
			\newblock {\em Ann. of Math. (2)}, 122(3):509--539, 1985.
			
			\bibitem{LS82}
			E.~Lehrer and B.~Weiss.
			\newblock An {$\varepsilon $}-free {R}ohlin lemma.
			\newblock {\em Ergodic Theory Dynam. Systems}, 2(1):45--48, 1982.
			
			\bibitem{LMY}
			C.~Liang, K.~Marin, and J.~Yang.
			\newblock Lyapunov exponents of partially hyperbolic volume-preserving maps
			with 2-dimensional center bundle.
			\newblock {\em Ann. Inst. H. Poincar\'e{} C Anal. Non Lin\'eaire},
			35(6):1687--1706, 2018.
			
			\bibitem{malicet17}
			D.~Malicet.
			\newblock Random walks on {${\rm Homeo}(S^1)$}.
			\newblock {\em Comm. Math. Phys.}, 356(3):1083--1116, 2017.
			
			\bibitem{M}
			K.~Marin.
			\newblock {$C^r$}-density of (non-uniform) hyperbolicity in partially
			hyperbolic symplectic diffeomorphisms.
			\newblock {\em Comment. Math. Helv.}, 91(2):357--396, 2016.
			
			\bibitem{Martinchich}
			S.~Martinchich.
			\newblock Global stability of discretized {A}nosov flows.
			\newblock {\em J. Mod. Dyn.}, 19:561--623, 2023.
			
			\bibitem{OP}
			D.~Obata and M.~Poletti.
			\newblock Positive exponents for random products of conservative surface
			diffeomorphisms and some skew products.
			\newblock {\em J. Dynam. Differential Equations}, 34(3):2405--2428, 2022.
			
			\bibitem{Ose68}
			V.~I. Oseledec.
			\newblock A multiplicative ergodic theorem. {C}haracteristic {L}japunov,
			exponents of dynamical systems.
			\newblock {\em Trudy Moskov. Mat. Ob\v s\v c.}, 19:179--210, 1968.
			
			\bibitem{notasm}
			M.~Poletti.
			\newblock Stably positive {L}yapunov exponents for symplectic linear cocycles
			over partially hyperbolic diffeomorphisms.
			\newblock {\em Discrete Contin. Dyn. Syst.}, 38(10):5163--5188, 2018.
			
			\bibitem{PSW}
			C.~Pugh, M.~Shub, and A.~Wilkinson.
			\newblock H\"older foliations.
			\newblock {\em Duke Math. J.}, 86(3):517--546, 1997.
			
			\bibitem{S}
			R.~Saghin.
			\newblock On invariant holonomies between centers.
			\newblock {\em Ergodic Theory Dynam. Systems}, 45(1):274--293, 2025.
			
			\bibitem{Sin72}
			J.~G. Sinai.
			\newblock Gibbs measures in ergodic theory.
			\newblock {\em Uspehi Mat. Nauk}, 27(4(166)):21--64, 1972.
			
			\bibitem{tahzibi-yang-IP}
			A.~Tahzibi and J.~Yang.
			\newblock Invariance principle and rigidity of high entropy measures.
			\newblock {\em Trans. Amer. Math. Soc.}, 371(2):1231--1251, 2019.
			
			\bibitem{Almost}
			M.~Viana.
			\newblock Almost all cocycles over any hyperbolic system have nonvanishing
			{L}yapunov exponents.
			\newblock {\em Ann. of Math. (2)}, 167(2):643--680, 2008.
			
			\bibitem{Viana-Yang}
			M.~Viana and J.~Yang.
			\newblock Physical measures and absolute continuity for one-dimensional center
			direction.
			\newblock {\em Ann. Inst. H. Poincar\'e{} C Anal. Non Lin\'eaire},
			30(5):845--877, 2013.
			
			\bibitem{Yan21}
			J.~Yang.
			\newblock Entropy along expanding foliations.
			\newblock {\em Adv. Math.}, 389:Paper No. 107893, 39, 2021.
			
		\end{thebibliography}
	}
\end{document}